\documentclass[11pt,reqno]{amsart}
\usepackage{amscd, amsfonts, amsmath, amssymb, amsthm, mathrsfs, appendix, enumerate, epsfig, tabu, graphicx, hyperref, pdfpages, tikz, mathtools, comment, tikz-cd,pinlabel,bbm,subcaption}
\usepackage[all,cmtip]{xy}
\makeatletter
\@namedef{subjclassname@2020}{Mathematics Subject Classification \textup{2020}}

\newsavebox{\@brx}
\newcommand{\llangle}[1][]{\savebox{\@brx}{\(\m@th{#1\langle}\)}%
	\mathopen{\copy\@brx\kern-0.5\wd\@brx\usebox{\@brx}}}
\newcommand{\rrangle}[1][]{\savebox{\@brx}{\(\m@th{#1\rangle}\)}%
	\mathclose{\copy\@brx\kern-0.5\wd\@brx\usebox{\@brx}}}

\makeatother

\newtheorem{theorem}{Theorem}[section]
\newtheorem{corollary}[theorem]{Corollary}
\newtheorem{lemma}[theorem]{Lemma}
\newtheorem{proposition}[theorem]{Proposition}

\theoremstyle{definition}

\newtheorem{definition}[theorem]{Definition}
\newtheorem{example}[theorem]{Example}
\newtheorem{remark}[theorem]{Remark}
\numberwithin{equation}{subsection}
\newtheorem*{ack}{Acknowledgement}

\newcommand{\Conj}{\operatorname{Conj}}
\newcommand{\Core}{\operatorname{Core}}

\newcommand{\Hom}{\operatorname{Hom}}

\newcommand{\Map}{\mathrm{Map}}

\newcommand{\id}{\mathrm{id}}

\newcommand{\ie}{\mathrm{i.e.}}

\begin{document}
	\setcounter{section}{0}
	\title{Automorphisms, cohomology and extensions of symmetric quandles}
	\author{Biswadeep Karmakar}	
	\author{Deepanshi Saraf}
	\author{Mahender Singh}
	
	\address{Department of Mathematical Sciences, Indian Institute of Science Education and Research (IISER) Mohali, Sector 81,  S. A. S. Nagar, P. O. Manauli, Punjab 140306, India.}
	\email{biswadeep.isi@gmail.com}
	\email{saraf.deepanshi@gmail.com}
	\email{mahender@iisermohali.ac.in}
	
	
	
\begin{abstract}
It is well-known that the cohomology of symmetric quandles generates robust cocycle invariants for unoriented classical and surface links. Expanding on the recently introduced module-theoretic generalized cohomology for symmetric quandles, we derive a four-term exact sequence that relates 1-cocycles, second cohomology, and a specific group of automorphisms associated with the extensions of symmetric quandles. This exact sequence shows that the obstruction to lifting and extending automorphisms is found in the second symmetric quandle cohomology. Additionally, some general aspects of dynamical cocycles and extensions are discussed.
\end{abstract}
\maketitle
	
\section{Introduction}
Quandles are algebraic structures that systematically represent the three Reidemeister moves in planar diagrams of oriented links within the 3-sphere. When extended to racks, they become strong algebraic tools for addressing codimension two embeddings in any $n$-sphere \cite{MR2255194}. Their study gained significant momentum following the seminal works of Joyce \cite{Joyce1979} and Matveev \cite{MR0672410}. They showed that link quandles are complete invariants for non-split links, considering the orientation of the ambient space.
\par

In classical knot theory, having an orientation is often crucial, especially when dealing with quandle cocycle invariants. Kamada and Oshiro  \cite{MR2371714, MR2657689} addressed the challenge of eliminating the need for orientation by introducing symmetric quandles and their homology groups. They defined a quandle cocycle invariant for an unoriented link using a given symmetric quandle 2-cocycle, and showed that aligns with the quandle cocycle invariant of an oriented link when an orientation is arbitrarily assigned. Similarly, an invariant for unoriented surface links has been defined using symmetric quandle 3-cocycles, yielding comparable results. Using 4-fold symmetric quandles, an invariant for 3-manifolds has been defined \cite{MR2945646, MR2821435}, and it has been proved that the Chern--Simons invariant of closed 3-manifolds can be interpreted as a quandle cocycle invariant. In \cite{MR2890467}, symmetric quandles have also been employed to define invariants for spatial graphs and handlebody-links. These results, along with further refinements (see, for example  \cite{MR4476068, MR2629767, MR4594919}), underscore the significance of symmetric quandles in low-dimensional topology.
\par

Motivated by \cite{MR1994219, MR2155522}, in \cite{KSS2024}, we examined these objects from a categorical perspective and determined the Beck modules  \cite{MR2616383} in the category of symmetric racks (respectively, quandles). Beck modules provide a quite general concept of a coefficient module for an associated (co)homology theory. They generalize the notion of coefficient modules from group cohomology, Lie algebra cohomology, and Hochschild cohomology of associative algebras. We formulated a generalised (co)homology theory for symmetric racks (respectively, quandles) with homogeneous coefficients, extending the cohomolgy theory introduced in \cite{MR2657689}.
\par

In this paper, we develop this extension theory further, and establish an exact sequence connecting automorphisms of symmetric quandles with their extensions and the generalised second cohomology. This work builds upon the paper \cite{MR4282648} by extending the results to symmetric racks and quandles, and utilizing more general module-theoretic coefficients developed in \cite{KSS2024}. The paper is organised as follows. In Section \ref{sectionpreliminaries}, some necessary preliminaries have been recalled. In Section \ref{section dynamical cocycles}, we introduce dynamical cocycles and extensions of symmetric racks and quandles (Proposition \ref{Dynamical cocycles for symmetric racks} and Proposition \ref{equivalence of dynamical extensions}). In Section \ref{sec dynamical cocycle and second cohomology}, we relate dynamical cocycles and ordinary cocycles of symmetric racks and quandles (Theorem \ref{dynamical cocycles vs symmetric cocycles}). In Section \ref{sec group extensions and dynamical extensions}, we prove that group extensions lead to some natural dynamical extensions of symmetric racks and quandles (Corollary \ref{conj core extension} and Corollary \ref{core extension}).  Finally, in Section \ref{sectionwellstype}, we prove our main result (Theorem \ref{MainTheoremWells}). More precisely, we prove that if $(X, \rho_X)$ is a symmetric rack,  $\mathcal{F}=(A, \phi, \psi, \eta)$ is an $(X, \rho_{X})$-module, $\alpha \in Z_{S R}^2((X, \rho_X); A)$ is a 2-cocycle and $E(\mathcal{F}, \alpha)$ is the abelian extension of $(X, \rho_X)$ by $\mathcal{F}$ through $\alpha$, then there is an exact sequence 
$$
0\to Z_{SR}^{1} \big((X, \rho_{X});A \big) \to \operatorname{Aut}_{A} \big(E(\mathcal{F}, \alpha), \rho_{E(\mathcal{F}, \alpha)} \big) \stackrel{\Gamma}{\to} \operatorname{Aut} (X, \rho_{X})\times\operatorname{Aut}(A) \stackrel{\Lambda_{[\alpha]}}{\to}H_{SR}^{2} \big((X,\rho_{X});A \big) 
$$
of groups. The exact sequence shows that the obstruction to lifting and extension of automorphisms lies in  the second symmetric rack cohomology. There is a similar exact sequence for symmetric quandles as well.
\medskip

\section{Preliminaries} \label{sectionpreliminaries}
In this section, we quickly recall basic terminology that we shall need in latter sections. We refer the reader to \cite{KSS2024} for details. Throughout, all our racks and quandles will be right-distributive.
	\begin{definition}
		Let $(X,*)$ be a rack (respectively, quandle). A map $\rho_X:X \rightarrow X$ is called a \textit{good involution} if 
		\begin{enumerate}
			\item[(S1)] $\rho_X^2=\id_X$,
			\item[(S2)] $\rho_X(x \ast y)= \rho_X(x) \ast y$,
			\item[(S3)] $x \ast \rho_X(y) = x \ast^{-1}y$,
		\end{enumerate}
		for all $x, y \in X$.  The pair $(X, \rho_X)$ is referred as a \textit{symmetric rack (respectively, quandle)}.
	\end{definition}

We refer the reader to \cite{MR2371714,MR3363811, MR2657689,MR4594919} for more details on these structures and their applications in knot theory.

\begin{example}
Below we list some elementary examples:
\begin{enumerate}
\item Every involution on a trivial quandle is a good involution.
\item For each $n \ge 1$, if $X= \Conj_n(G)$ is the $n$-conjugation quandle of a group $G$, then $\rho_X:X \rightarrow X$  given by $\rho_X(x)=x^{-1}$ is a good involution. We denote this symmetric quandle by $(\Conj_n,\mathrm{inv}).$
			\item If $X= \Core(G)$ is the core quandle of a group $G$, then the identity map $\id:X \rightarrow X$ is a good involution. Also, if $z \in G$ is a central element of order two, then $\rho_X:X \rightarrow X$ given by $\rho(y)=y z$ is another good involution of $X$. We denote these symmetric quandles by $(\Core(G),\id)$ and $(\Core(G),z)$, respectively.
		\end{enumerate}
	\end{example}	
	
\begin{definition}
A \textit{symmetric rack (respectively, quandle) homomorphism} is a map $f:(X,\rho_X) \rightarrow (Y, \rho_Y)$ such that 
$$f(x \ast y)=f(x) \ast f(y) \quad \textrm{and} \quad f \big(\rho_X(x) \big)=\rho_Y \big(f(x) \big)$$ for all $x,y \in X$. A bijective symmetric rack (respectively, quandle) homomorphism is called a \textit{symmetric rack (respectively, quandle) isomorphism}.
	\end{definition}
	
Let us recall the definition of a module over a symmetric rack, introduced in \cite[Definition 3.1]{KSS2024} through the language of trunks.

\begin{definition}
Let $(X,\rho_X)$ be a symmetric rack. An {\it $(X,\rho_X)$-module}, denoted by $\mathscr{F}=(A,\phi,\psi,\eta)$, is a collection of abelian groups $\{A_x \mid x\in X\}$ together with group isomorphisms $\phi_{x,y}:A_x \rightarrow A_{x \ast y}$, and group homomorphisms $\psi_{x,y}:A_y \rightarrow A_{x \ast y}$ and $\eta_x:A_x \rightarrow A_{\rho_X(x)}$ such that 
	\begin{enumerate}[(M1)]
		\item \label{M1} $\phi_{x \ast y,z} \,\phi_{x,y}=\phi_{x \ast z,y \ast z}\,\phi_{x,z}$,
		\item \label{M2} $\phi_{x \ast y,z}\,\psi_{x,y}=\psi_{x \ast z,y \ast z}\,\phi_{y,z}$,
		\item \label{M3} $\eta_{\rho_X(x)} \,\eta_x = \id$,
		\item \label{M4} $\eta_{x \ast y} \,\phi_{x,y}= \phi_{\rho_X(x),y} \,\eta_x$,
		\item \label{M5} $\psi_{\rho_X(x),y}=\eta_{x \ast y} \,\psi_{x,y}$,
		\item \label{M6}$\phi_{x \ast^{-1}y,y}\, \phi_{x, \rho_X(y)}=\id$,
	    \item \label{M7} $\psi_{x\ast y,z}(a)= \phi_{x\ast z,y\ast z}\, \psi_{x,z}(a)+\psi_{x\ast z,y\ast z} \,\psi_{y,z}(a)$,
		\item \label{M8} $\phi_{x\ast^{-1} y,y}\psi_{x,\rho_X(y)}(\eta_{y}(b))=-\psi_{x\ast \rho_X(y),y}(b)$
	\end{enumerate}
	hold for all $x,y, z \in X$, $a \in A_z$ and $b \in A_y$.
	\par

If $(X,\rho_X)$ is a symmetric quandle, then $\mathscr{F}$ additionally satisfies the condition
	\begin{enumerate}\label{M9}
		\item [(M9)] $\phi_{x,x}(a)+\psi_{x,x}(a)=a$
	\end{enumerate} 
	for all $x \in X$ and $a\in A_x$.
\end{definition}
	
Next, we recall the definition of a symmetric rack algebra \cite[Definition 6.1]{KSS2024}. 

\begin{definition}
		The \textit{symmetric rack algebra} of a symmetric rack $(X, \rho_X)$, denoted by $\mathbb{Z}(X,\rho_X)$, is the associative $\mathbb{Z}$-algebra which is generated by the set $\{\phi_{x,y}^{\pm 1}, \, \psi_{x,y}, \,\eta_x \mid x,y \in X \}$ and admits the following defining relations: 
		\begin{enumerate}\label{axiom}
			\item[(A1)] $\phi_{x,y} \, \phi_{x,y}^{-1}=\phi_{x,y}^{-1} \,\phi_{x,y}=1$,
			\item[(A2)] $\phi_{x \ast y,z} \,\phi_{x,y}=\phi_{x \ast z, y \ast z} \,\phi_{x,z}$,
			\item[(A3)] $\phi_{x \ast y,z} \,\psi_{x,y}=\psi_{x \ast z, y \ast z} \,\phi_{y,z}$,
			\item[(A4)] $\eta_{\rho_X(x)} \,\eta_x=1$,
			\item[(A5)] $\phi_{\rho_X(x),y} \,\eta_x= \eta_{x\ast y} \, \phi_{x,y}$,
			\item[(A6)] $\psi_{\rho_X(x),y}= \eta_{x \ast y} \, \psi_{x,y}$,
			\item[(A7)] $\phi_{x \ast \rho_X(y),y} \,\phi_{x, \rho_X(y)}=1$,
			\item[(A8)] $\psi_{x \ast y,z}=\phi_{x \ast z, y \ast z} \,\psi_{x,z}+\psi_{x \ast z, y \ast z}\, \psi_{y,z}$,
			\item[(A9)] $\phi_{x \ast \rho_X(y),y} \,\psi_{x, \rho_X(y)} \,\eta_y =- \psi_{x \ast \rho_X(y), y}$
		\end{enumerate}
	for all $x,y,z \in X.$
\par
The \textit{symmetric quandle algebra} of a symmetric quandle $(X, \rho_X)$ is the symmetric rack algebra of the underlying symmetric rack which satisfies the additional relation
		\begin{enumerate}\label{Qaxiom}
			\item [(A10)] $\phi_{x,x}+\psi_{x,x}=1$
		\end{enumerate}
	for all $x \in X.$
\end{definition}

	Let $(X,\rho_X)$ be a symmetric rack. For each $n\geq0$, define $C_n(X,\rho_X):=\mathbb{Z}(X,\rho_X)X^n,~\ie$, the free left $\mathbb{Z}(X,\rho_X)$-module with basis $X^n$, where $X^0=\{p\}$ for some fixed $p \in X$. For the sake of convenience, we fix the notation 
	$$[x_1 \cdots x_n]:=\big(((x_1 \ast x_2) \ast x_3) \ast \cdots \ast x_n\big)$$ for any $(x_1, x_2, \ldots, x_n) \in X^n$. For each $n \geq 2$, let
	$\partial_n:C_{n}(X,\rho_X) \rightarrow C_{n-1}(X, \rho_X)$ be the $\mathbb{Z}(X,\rho_X)$-linear map defined on the basis as
	\begin{eqnarray*}
		\partial_n \big((x_1, \ldots, x_n) \big) &= &\sum_{j=2}^{n}(-1)^j (-1)^n \phi_{[x_1 \cdots \widehat{x_j} \cdots x_n],[x_j \cdots  x_n]}
		(x_1,\ldots, \widehat{x_j}, \ldots , x_n)\\
		& - &\sum_{j=2}^{n}(-1)^j (-1)^n (x_1 \ast x_j,\ldots,x_{j-1} \ast x_j, x_{j+1}, \ldots , x_n)\\
		& + &(-1)^n \psi_{[x_1 \widehat{x_2}x_3 \cdots x_n],[x_2 x_3 \cdots x_n]}(x_2,\ldots , x_n)
	\end{eqnarray*}
	and 
	$$\partial_1(x)=-\psi_{x \ast^{-1} p,p}(p).$$
Then, $(C_*(X,\rho_X),\partial_*)$ is a chain complex.
\par   
	
Next, we define a subcomplex of this chain complex as follows. For each $n \ge 1$, consider the left $\mathbb{Z}(X,\rho_X)$-submodule $D_n^{SR}(X,\rho_X)$ of $C_n(X,\rho_X)$ generated by $U_n \cup V_n$, where
	$$U_n= \big\{ \eta_{[x_1 \cdots x_n]}(x_1, \ldots, x_n)- (\rho_X(x_1),x_2, \ldots ,x_n ) \mid (x_1, x_2, \ldots, x_n) \in X^n \big\},$$

	\begin{eqnarray*}
		V_n &=& \bigcup_{i=2}^n \Big\{\big(\phi_{[x_1 \cdots \widehat{x_i} \cdots x_n],[x_i \cdots x_n]} (x_1 \ast x_i, \ldots , x_{i-1} \ast x_i, \rho_X(x_i),x_{i+1}, \ldots, x_n)+(x_1, \ldots, x_n) \big) \mid\\
		&&  (x_1, x_2, \ldots, x_n) \in X^n \Big\}.
	\end{eqnarray*}
for $n \ge 2$ and $V_1=\emptyset$. Then, by  \cite[Proposition 6.6]{KSS2024}, we have $\partial_n(D_{n}^{SR}(X,\rho_X)) \subset D_{n-1}^{SR}(X,\rho_X)$ for each $n \ge 1$.
\par

Set $C_n^{SR}(X,\rho_X)= C_n(X,\rho_X)/ D_n^{SR}(X,\rho_X)$, and let $A$ be a left $\mathbb{Z}(X,\rho_X)$-module. Define $$C^n_{SR} \big((X,\rho_X);A \big):=\Hom_{\mathbb{Z}(X,\rho_X)} \big(C_n^{SR}(X,\rho_X),A\big)$$ and let $\delta^n: C^n_{SR}((X,\rho_X);A) \to C^{n+1}_{SR}((X,\rho_X);A)$ be the induced coboundary map. Then, $\big(C^*_{SR}((X,\rho_X);A),\delta^*\big)$ is a cochain complex, leading to cohomology groups $H^*_{SR}((X,\rho_X);A)$. The group of $n$-cocycles and the group of $n$-coboundaries will be denoted by $Z_{SR}^{n}((X, \rho_{X});A)$ and $B_{SR}^{n}((X, \rho_{X});A)$, respectively.
\par

We can modify the preceding construction for symmetric quandles. Let $(X,\rho_X)$ be a symmetric quandle.  Define $D_n^{SQ}(X, \rho_X)$ to be the submodule of $C_n(X,\rho_X)$ generated by the $U_n \cup V_n \cup W_n$, where 
	$$W_n:=\bigcup_{i=1}^{n-1} \Big\{(x_1, \ldots, x_n)\in X^n \mid x_i=x_{i+1} \Big\}.$$
By \cite[Proposition 6.8]{KSS2024}, we have $\partial_n(D^{SQ}_{n}(X,\rho_X)) \subset D^{SQ}_{n-1}(X,\rho_X)$ for each $n \ge 2$. We can now define the homology and cohomology of symmetric quandles. Set $C_n^{SQ}(X,\rho_X)= C_n(X,\rho_X)/ D_n^{SQ}(X,\rho_X)$, and let $A$ be a left $\mathbb{Z}(X,\rho_X)$-module. We define $$C^n_{SQ} \big((X,\rho_X);A \big):=\Hom_{\mathbb{Z}(X,\rho_X)}\big(C_n^{SQ}(X,\rho_X),A \big)$$ and take $\delta^n$ to be the same map as in the case of symmetric racks. Then $\big(C^*_{SQ}((X,\rho_X);A),\delta^*\big)$ is a cochain complex yielding cohomology group $H^*_{SQ}((X,\rho_X);A)$. As before, the group of $n$-cocycles and the group of $n$-coboundaries will be denoted by $Z_{SQ}^{n}((X, \rho_{X});A)$ and $B_{SQ}^{n}((X, \rho_{X});A)$, respectively.

	\begin{remark} \label{cocycleconditions}
It follows from the  definition that, a map $\sigma: C_2^{SR}(X,\rho_X) \rightarrow A$, where $A$ is a left $\mathbb{Z}(X,\rho_X)$-module, is a symmetric rack 2-cocycle if and only if $\sigma$ satisfies the following conditions:
\begin{enumerate}
\item \label{cocyclecondition1}  $-\phi_{x *z,y *z}\sigma(x,z)+\phi_{x* y,z}\sigma(x,y)+\sigma(x*y, z)-\sigma(x *z, y*z)-\psi_{x* z,y* z}\sigma(y,z)=0$,
\item \label{cocyclecondition2.1} $\eta_{x *y}\sigma(x,y)-\sigma(\rho_X(x),y)=0$,
\item \label{cocyclecondition2.2} $\phi_{x,y}\sigma(x*y,\rho_X(y))+\sigma(x,y)=0$
\end{enumerate}
for all $x,y,z \in X$.
Furthermore,  $\sigma$ is a symmetric quandle 2-cocycle if it additionally satisfies $\sigma(x,x)=0$ for all $x \in X$.
\end{remark}

\begin{remark}
Before proceeding further, we set some notation for convenience.
\begin{itemize}
\item For maps $f:X \times X \to A$ and $g:X \to A$, we shall sometimes write $f_{x, y}$ and $g_x$ to denote $f(x, y)$ and $g(x)$, respectively.
\item In sections \ref{section dynamical cocycles}, \ref{sec dynamical cocycle and second cohomology} and \ref{sec group extensions and dynamical extensions}, we shall write a symmetric rack (respectively, quandle) as $(X, \rho)$ instead of $(X, \rho_X)$.
\end{itemize}
\end{remark}
\medskip

\section{Dynamical cocycles for symmetric quandles}\label{section dynamical cocycles}
A far-reaching generalization of quandle extensions using dynamical 2-cocycles has been provided in  \cite[Lemma 2.1]{MR1994219}. Below we present a more general and symmetric analogue of this construction.

\begin{proposition}\label{Dynamical cocycles for symmetric racks}
Let $(X,\rho)$ be a symmetric rack and $S$ be a collection of sets $\{S_x \mid x \in X\}$. Let $\alpha:X \times X \rightarrow \cup_{x,y \in X}\Map\big(S_x \times S_y,S_{x*y}\big)$  and $\beta:X \rightarrow \cup_{x \in X}\Map \big(S_x,S_{\rho(x)}\big)$ be two maps such that $\alpha({x,y}) \in \Map \big(S_x \times S_y, S_{x*y}\big)$ and $\beta(x) \in \Map\big(S_x,S_{\rho(x)}\big)$. We denote the image of $\alpha(x,y)(s,t)$  and $\beta(x)(s)$ by $\alpha_{x,y}(s,t)$ and $\beta_x(s)$, respectively. Then, the set $X \times S:=\{(x,s) \mid x \in X, \, s \in S_x\}$ is a symmetric rack with the binary operation
		$$(x,s)*(y,t)=\big(x*y,\alpha_{x,y}(s,t) \big)$$
		and the good involution $\rho_{\alpha,\beta}:X \times S \rightarrow X \times S$ given by
		\begin{equation*}
			\rho_{\alpha,\beta} \big((x,s) \big) = \big(\rho(x),\beta_x(s) \big)
		\end{equation*}
		if and only if the following conditions hold for all $x,y \in X$ and $s\in S_x, t \in S_y, w\in S_z$:
		\begin{enumerate}
			\item $\alpha_{x,y}(\textunderscore,t):S_x \rightarrow S_{x*y}$ is a bijection.
			\item $\alpha_{x*y,z} \big(\alpha_{x,y}(s,t),w \big)=\alpha_{x*z,y*z} \big(\alpha_{x,z}(s,w),\alpha_{y,z}(t,w) \big).$
			\item $\alpha_{\rho(x),y} \big(\beta_x(s),t \big)=\beta_{x*y} \big(\alpha_{x,y}(s,t) \big).$
			\item $\beta_{\rho(x)}\beta_x(s)=s.$
			\item $\alpha_{x,\rho(y)}(\beta_y(t))(s)=(\alpha_{x*^{-1}y,y}(t))^{-1}(s)$, where $\alpha_{x,y}(t)(s)=\alpha_{x,y}(s,t).$
		\end{enumerate}
	Furthermore, if $(X,\rho)$ is a symmetric quandle, then  $X \times S$ is a symmetric quandle if and only if $\alpha$ satisfies the additional condition: 
	\begin{enumerate}
		\item [(6)] $\alpha_{x,x}(s,s)=s$ for all $x \in X, s \in S_x.$
	\end{enumerate}
	
	\end{proposition}
\begin{proof}
	The proof is a straightforward calculation.
\end{proof}

Note that, if $S$ contains a single set, then $X \times S$ is just the usual cartesian product.

 \begin{definition}
Let $(X, \rho)$ be a symmetric rack (respectively, quandle).  
\begin{enumerate}
\item The symmetric rack (respectively, quandle) obtained in Proposition \ref{Dynamical cocycles for symmetric racks} is called an {\it extension} of $X$ by $S$ through $(\alpha,\beta)$, and is denoted by $\big(X \times_{(\alpha,\beta)}S,\rho_{\alpha,\beta}\big).$
\item  The pair $(\alpha,\beta)$ satisfying the conditions in Proposition \ref{Dynamical cocycles for symmetric racks} is called a {\it dynamical cocycle} for the symmetric rack (respectively, quandle) $(X,\rho)$ over $S$. 
\end{enumerate}
 \end{definition}

 The natural projection $p:\big(X \times_{(\alpha,\beta)}S,\rho_{\alpha,\beta}\big) \rightarrow (X,\rho)$ given by $(x,s) \mapsto x$ is a morphism of symmetric racks (respectively, quandles). Next, we consider equivalence of these extensions.

\begin{definition}
Let $(X,\rho)$ be a symmetric rack (respectively, quandle), and $(\alpha,\beta)$ and $(\alpha',\beta')$ be dynamical cocycles of $(X,\rho)$ over $S$. We say that the extensions $\big(X \times_{(\alpha,\beta)}S,\rho_{\alpha,\beta}\big)$ and $\big(X \times_{(\alpha',\beta')}S,\rho_{\alpha',\beta'}\big)$ are {\it equivalent} if there is an isomorphism $f:\big(X \times_{(\alpha,\beta)}S,\rho_{\alpha,\beta}\big) \rightarrow \big(X \times_{(\alpha',\beta')}S,\rho_{\alpha',\beta'}\big)$ of symmetric racks (respectively, quandles) such that $p \, f =p.$
\end{definition}

Throughout, let $\Sigma_X$ denote the group of permutations of a set $X$.

\begin{definition}\label{dynamical cohomlogous}
Let $(\alpha,\beta)$ and $(\alpha',\beta')$ be  dynamical cocycles of a symmetric rack $(X,\rho)$ over $S$. We say that $(\alpha,\beta)$ and $(\alpha',\beta')$ are {\it cohomologous} if there exists a map $\gamma:X \rightarrow \cup_{x \in X}\Sigma_{S_x}$ with $\gamma(x) \in \Sigma_{S_x}$ such that 
	\begin{equation}\label{cohomologous dynamical cocycle}
		\alpha'_{x,y}(s,t):= \gamma_{x*y}\big(\alpha_{x,y}(\gamma_x^{-1}(s),\gamma_y^{-1}(t))\big)
		\text{ and }
		\beta'_x(s):=\gamma_{\rho(x)}\big(\beta_x(\gamma_x^{-1}(s))\big)
	\end{equation}
for all $x,y \in X$, $s\in S_x$ and $t \in S_y$. 
\end{definition}

\begin{proposition}\label{equivalence of dynamical extensions}
	Let $(\alpha,\beta)$ be a dynamical cocycle of a symmetric rack (respectively, quandle) $(X,\rho)$ over $S$ and $\gamma:X \rightarrow \cup_{x \in X}\Sigma_{S_x}$ be a map such that $\gamma(x) \in \Sigma_{S_x}$. Let
	\begin{equation*}
		\alpha'_{x,y}(s,t):= \gamma_{x*y}\big(\alpha_{x,y}(\gamma_x^{-1}(s),\gamma_y^{-1}(t))\big)
		\text{ and }
		\beta'_x(s):=\gamma_{\rho(x)}\big(\beta_x(\gamma_x^{-1}(s))\big)
	\end{equation*}
for all $x,y \in X$, $s\in S_x$ and $t \in S_y$. Then $(\alpha',\beta')$ is also a dynamical cocycle for $(X,\rho)$. Furthermore, 
$\big(X \times_{(\alpha,\beta)}S,\rho_{\alpha,\beta}\big)$ is isomorphic to $\big(X \times_{(\alpha',\beta')}S,\rho_{\alpha',\beta'}\big)$ as extensions of $(X,\rho)$ by $S$ if and only if $(\alpha',\beta')$ and $(\alpha, \beta)$ are cohomologous.
\end{proposition}
\begin{proof}
	First, we verify that $(\alpha',\beta')$ is a dynamical cocycle.
	\begin{enumerate}
		\item Clearly, $\alpha'_{x,y}(\textunderscore,t):= \gamma_{x*y}(\alpha_{x,y}(\textunderscore,\gamma_y^{-1}(t)))$ is a bijection, since $\alpha_{x,y}(\textunderscore,t)$ and $\gamma_y$ are so.
\item For all $x,y,z \in X$ and $s\in S_x,t \in S_y, w\in S_z $, we have
\begin{eqnarray*}
\alpha'_{x*y,z} \big(\alpha'_{x,y}(s,t),w \big)&=&\alpha'_{x*y,z} \big(\gamma_{x*y}(\alpha_{x,y}(\gamma_x^{-1}(s),\gamma_y^{-1}(t))),w \big)\\
&=&\gamma_{x*y*z} \big(\alpha_{x*y,z} \big(\alpha_{x,y}(\gamma_x^{-1}(s),\gamma_y^{-1}(t)),\gamma_z^{-1}(w) \big) \big),
\end{eqnarray*}
\begin{eqnarray*}
\alpha'_{x*z,y*z} \big(\alpha'_{x,z}(s,w),\alpha'_{y,z}(t,w) \big)&=& 	\alpha'_{x*z,y*z}\big(\gamma_{x*z}(\alpha_{x,z}(\gamma_x^{-1}(s),\gamma_z^{-1}(w))),\gamma_{y*z}(\alpha_{y,z}(\gamma_y^{-1}(t),\gamma_z^{-1}(w))) \big)\\
		&=&\gamma_{x*y*z} \big(\alpha_{x*z,y*z} \big(\alpha_{x,z}(\gamma_x^{-1}(s),\gamma_z^{-1}(w)),\alpha_{y,z}(\gamma_y^{-1}(t),\gamma_z^{-1}(w)) \big)\big).
	\end{eqnarray*}
Since $(\alpha,\beta)$ is dynamical cocycle, by condition $(2)$ in Proposition \ref{Dynamical cocycles for symmetric racks}, we see that $\alpha'_{x*y,z}(\alpha'_{x,y}(s,t),w)=\alpha'_{x*z,y*z}(\alpha'_{x,z}(s,w),\alpha'_{y,z}(t,w))$, which is desired.
\item For all $x,y \in X$ and $s\in S_x,t \in S_y $, we have
\begin{eqnarray*}	
	\alpha'_{\rho(x),y} \big(\beta'_x(s),t \big)&=&\alpha'_{\rho(x),y} \big(\gamma_{\rho(x)}(\beta_x(\gamma_x^{-1}(s))),t \big)\\
	&=&\gamma_{\rho(x)*y} \big(\alpha_{\rho(x),y} \big(\beta_x(\gamma_x^{-1}(s)),\gamma_y^{-1}(t) \big) \big)
\end{eqnarray*}	
and
\begin{eqnarray*}	
	\beta'_{x*y} \big(\alpha'_{x,y}(s,t) \big)&=&\beta'_{x*y} \big(\gamma_{x*y}\big(\alpha_{x,y}(\gamma_x^{-1}(s),\gamma_y^{-1}(t)) \big) \big)\\
	&=&\gamma_{\rho(x*y)} \big(\beta_{x*y}\big(\alpha_{x,y}(\gamma_x^{-1}(s),\gamma_y^{-1}(t))\big)\big).
\end{eqnarray*}
Since $(\alpha,\beta)$ is dynamical cocycle, by condition $(3)$ in Proposition \ref{Dynamical cocycles for symmetric racks}, we get 
$\alpha'_{\rho(x),y}(\beta'_x(s),t)=\beta'_{x*y}(\alpha'_{x,y}(s,t))$.
\item For all $x \in X$ and $s \in S_x$, we have
\begin{eqnarray*}
	\beta'_{\rho(x)}\beta'_x(s)&=&\beta'_{\rho(x)} \big(\gamma_{\rho(x)}(\beta_x(\gamma_x^{-1}(s))) \big)\\
	&=&\gamma_{\rho^2(x)}\big(\beta_{\rho(x)}(\gamma_{\rho(x)}^{-1}(\gamma_{\rho(x)}(\beta_x(\gamma_x^{-1}(s)))))\big)\\
	&=&\gamma_{x} \big(\beta_{\rho(x)}(\beta_x(\gamma_x^{-1}(s)))\big).
\end{eqnarray*}
Since $(\alpha,\beta)$ is dynamical cocycle, by condition $(4)$ in Proposition \ref{Dynamical cocycles for symmetric racks}, we have $\beta'_{\rho(x)}\beta'_x(s)=\gamma_{x}(\gamma_x^{-1}(s))=s.$
\item Note that $\big(\alpha'_{x,y}(t) \big)^{-1}(s)=\gamma_x \big((\alpha_{x,y}(\gamma_y^{-1}(t)))^{-1}(\gamma_{x*y}^{-1}(s))\big)$. Thus, for  $x,y \in X$ and $s\in S_x,t \in S_y$, we have
\begin{eqnarray*}
	\alpha'_{x,\rho(y)}(\beta'_y(t))(s)&=&\alpha'_{x,\rho(y)} \big(\gamma_{\rho(y)}(\beta_y(\gamma_y^{-1}(t)))\big)(s)\\
	&=&\gamma_{x*\rho(y)} \big(\alpha_{x,\rho(y)} \big(\gamma_x^{-1}(s),\gamma_{\rho(y)}^{-1}(\gamma_{\rho(y)}(\beta_y(\gamma_y^{-1}(t)))) \big)\big)\\
	&=&\gamma_{x*\rho(y)} \big(\alpha_{x,\rho(y)} \big(\gamma_x^{-1}(s),(\beta_y(\gamma_y^{-1}(t))) \big)\big)\\
	&=&\gamma_{x*\rho(y)} \big(\alpha_{x,\rho(y)}(\beta_y(\gamma_y^{-1}(t)))(\gamma_x^{-1}(s))\big)\\
	&=&\gamma_{x*\rho(y)}\big((\alpha_{x*^{-1}y,y}(\gamma_y^{-1}(t)))^{-1}(\gamma_x^{-1}(s))\big), \quad \textrm{since}~\alpha~\textrm{satisfies $(5)$ in Proposition \ref{Dynamical cocycles for symmetric racks}}\\
	&=&(\alpha'_{x*^{-1}y,y}(t))^{-1}(s).
\end{eqnarray*}
\item If $(X,\rho)$ is a symmetric quandle, then $\alpha$ satisfies condition $(6)$ in Proposition \ref{Dynamical cocycles for symmetric racks}. Thus, $$\alpha'_{x,x}(s,s)=\gamma_{x*x}\big(\alpha_{x,x}(\gamma_x^{-1}(s),\gamma_x^{-1}(s))\big)=\gamma_{x}\big(\gamma_x^{-1}(s)\big)=s$$ for all $x \in X$ and $s \in S_x$. 
\end{enumerate}
\par

Suppose that $(\alpha',\beta')$ and $(\alpha, \beta)$ are cohomologous. Then, there exists $\gamma$ such that  \eqref{cohomologous dynamical cocycle} holds. Define
$T:(X \times_{(\alpha,\beta)}S,\rho_{\alpha,\beta}) \rightarrow (X \times_{(\alpha',\beta')}S,\rho_{\alpha',\beta'})$ by $T(x,s) = (x, \gamma_x(s)).$ Clearly, $T$ is a bijection. Further,
$$T\big((x,s)*(y,t)\big)=T\big((x*y,\alpha_{x,y}(s,t))\big)=\big(x*y,\gamma_{x*y}(\alpha_{x,y}(s,t))\big)$$ and 
$$T\big((x,s)\big)*T \big((y,t) \big)=\big(x,\gamma_x(s)\big)*\big(y,\gamma_y(t)\big)=\big(x*y,\alpha'_{x, y}(\gamma_x(s),\gamma_{y}(t))\big).$$
Thus, by the definition of $\alpha'$, we have $T((x,s)*(y,t))=T((x,s))*T((y,t)).$  Similarly, 
$$T\big(\rho_{\alpha,\beta}(x,s) \big)=T \big(\rho(x),\beta_x(s) \big)=\big(\rho(x),\gamma_{\rho(x)}(\beta_x(s))\big)$$ and
$$\rho_{\alpha',\beta'}\big(T(x,s)\big)=\rho_{\alpha',\beta'} \big(x,\gamma_{x}(s)\big)=\big(\rho(x),\beta_x'(\gamma_{x}(s))\big).$$
Thus, by the definition of $\beta'$, we obtain $T(\rho_{\alpha,\beta}(x,s))=\rho_{\alpha',\beta'}(T(x,s)).$ Since $p \, T=p$,  we obtain $(X \times_{(\alpha,\beta)}S,\rho_{\alpha,\beta})\cong (X \times_{(\alpha',\beta')}S,\rho_{\alpha',\beta'})$ as extensions of $(X,\rho)$ by $S$. Conversely, suppose that $R: \big(X \times_{(\alpha,\beta)}S,\rho_{\alpha,\beta}\big) \to \big(X \times_{(\alpha',\beta')}S,\rho_{\alpha',\beta'}\big)$ is an isomorphism of extensions. Then, there exists a map $\gamma:X \rightarrow \cup_{x \in X} \Sigma_{S_{x}}$ such that $R(x,s) = (x,\gamma_x(s))$.  The fact that $R$ is a symmetric rack homomorphism implies that  $(\alpha,\beta)$ and $(\alpha',\beta')$ are related as in \eqref{cohomologous dynamical cocycle}, and hence are cohomologous. 
\end{proof}
\medskip

\section{Dynamical cocycles and second cohomology}\label{sec dynamical cocycle and second cohomology}
	Let $(X,\rho)$ be a symmetric rack. Recall from Remark \ref{cocycleconditions} that a map $\sigma:X \times X \rightarrow \cup_{x \in X} A_{x}$ is a 
symmetric rack 2-cocycle if and only if $\sigma$ satisfies:
\begin{enumerate}\label{cocycle condition}
	\item [$\mathrm{(F1)}$]$\sigma_{x \ast y,z}+\phi_{x \ast y,z}(\sigma_{x,y})=\phi_{x \ast z, y \ast z}(\sigma_{x,z})+\sigma_{x \ast z, y \ast z}+\psi_{x \ast z, y \ast z}(\sigma_{y,z})$,
	\item[$\mathrm{(F2)}$] $\sigma_{\rho_X(x),y}=\eta_{x\ast y}(\sigma_{x,y})$,
	\item[$\mathrm{(F3)}$] $\phi_{x\ast \rho_X(y),y}(\sigma_{x,\rho_X(y)})=-\sigma_{x \ast \rho_X(y),y}$
\end{enumerate}	
for all $x,y,z \in X$, where $\sigma(x,y)=\sigma_{x,y}$ and $\sigma_{x,y} \in A_{x*y}$. Also if $(X,\rho)$ is a symmetric quandle, then $\sigma$ is  a symmetric quandle 2-cocycle if in addition, we have $\sigma_{x,x}=0$ for all $x \in X$. The following theorem extends \cite[Proposition 2.25]{MR1994219}.
 
\begin{theorem}\label{dynamical cocycles vs symmetric cocycles}
		 Let $(X,\rho)$ be a symmetric rack and $A=\{A_x \mid x \in X\}$ be a collection of abelian groups. Consider isomorphisms $\phi_{x,y}:A_x \rightarrow A_{x*y}$, homomorphisms $\psi_{x,y}: A_y \rightarrow A_{x*y}$ and $\eta_x: A_x \rightarrow A_{\rho(x)}$, and a map $\sigma:X \times X \rightarrow \cup_{x \in X} A_{x}$ such that $\sigma_{x,y} \in A_{x*y}$ for all $x, y \in X$. Define
$$ \alpha : X \times X \rightarrow \cup_{x,y \in X}\Map \big( A_x \times A_y, A_{x*y}\big)~~\textrm{by}~~\alpha_{x,y}(a,b) = \phi_{x,y}(a)+\psi_{x,y}(b)+\sigma_{x,y} $$
and 
$$ \beta:X \rightarrow \cup_{x \in X}\Map\big( A_x,  A_{\rho(x)}\big)~\textrm{by}~\beta_x(a) = \eta_x(a)$$
		for all $x,y \in X$, $a \in A_x$ and $b \in A_y$.  Then the following assertions hold:
		\begin{enumerate}[(A)]
		\item The pair $(\alpha,\beta)$ is a dynamical cocycle for the symmetric rack $(X,\rho)$ over $A$ if and only if $(A,\phi_{x,y},\psi_{x,y},\eta_x)$ is an $(X,\rho)$-module and $\sigma$ is a symmetric rack 2-cocycle.
		\item[] Furthermore, if $(X,\rho)$ is a symmetric quandle, then the pair $(\alpha,\beta)$ is a dynamical cocycle for  $(X,\rho)$ over $A$ if and only if $(A,\phi_{x,y},\psi_{x,y},\eta_x)$ is an $(X,\rho)$-module and $\sigma$ is a symmetric quandle 2-cocycle.
		\item Let $\sigma'$ be another symmetric rack (respectively, quandle) $2$-cocycle and $(\alpha',\beta')$ its corresponding dynamical cocycle as defined above. If $\sigma$ and $\sigma'$ are cohomologous, then $(\alpha,\beta)$ and $(\alpha',\beta')$ are cohomologous.
	\end{enumerate}	
\end{theorem}
\begin{proof}
\begin{enumerate}[(A)]
\item We proceed case by case as follows:
\begin{enumerate}[(1)]
\item It is easy to see that $\alpha_{x,y}(b)$ is a bijection, with inverse defined as $\alpha_{x,y}(b)^{-1}(a)=\phi_{x,y}^{-1}(a)-\phi_{x,y}^{-1}\psi_{x,y}(b)-\phi_{x,y}^{-1}\sigma_{x,y}$ for all $x, y \in X$. 
			\item  For all $x,y,z \in X$ and $a \in A_x,b \in A_y,c \in A_z,$ we have
			\begin{eqnarray*}
&& \alpha_{x*y,z}(\alpha_{x,y}(a,b),c)\\
				&=&	\alpha_{x*y,z}(\phi_{x,y}(a)+\psi_{x,y}(b)+\sigma_{x,y},c)\\
				&=&\phi_{x*y,z}(\phi_{x,y}(a)+\psi_{x,y}(b)+\sigma_{x,y})+\psi_{x*y,z}(c)+\sigma_{x*y,z}
			\end{eqnarray*}
and		
\begin{eqnarray*}
&& \alpha_{x*z,y*z}(\alpha_{x,z}(a,c),\alpha_{y,z}(b,c))\\
				&=&\alpha_{x*z,y*z}(\phi_{x,z}(a)+\psi_{x,z}(c)+\sigma_{x,z},\phi_{y,z}(b)+\psi_{y,z}(c)+\sigma_{y,z})\\
				&=&\phi_{x*z,y*z}(\phi_{x,z}(a)+\psi_{x,z}(c)+\sigma_{x,z})+\psi_{x*z,y*z}(\phi_{y,z}(b)+\psi_{y,z}(c)+\sigma_{y,z})+\sigma_{x*z,y*z}.
			\end{eqnarray*}
Let $0_x$ denote the identity of the group $A_x$ for each $x \in X$. If $\alpha$ is a dynamical cocycle, then $\alpha_{x*y,z}(\alpha_{x,y}(a,b),c)=\alpha_{x*z,y*z}(\alpha_{x,z}(a,c),\alpha_{y,z}(b,c))$. Putting $a=0_x$, $b=0_y$ and $c=0_z$ in above equations gives
$$\phi_{x*y,z}(\sigma_{x,y})+\sigma_{x*y,z}=\phi_{x*z,y*z}(\sigma_{x,z})+\psi_{x*z,y*z}(\sigma_{y,z})+\sigma_{x*z,y*z}.$$
Thus, $\alpha_{x*y,z}(\alpha_{x,y}(a,b),c)=	\alpha_{x*z,y*z}(\alpha_{x,z}(a,c),\alpha_{y,z}(b,c))$ implies that $$\phi_{x*y,z}(\phi_{x,y}(a)+\psi_{x,y}(b))+\psi_{x*y,z}(c)=\phi_{x*z,y*z}(\phi_{x,z}(a)+\psi_{x,z}(c))+\psi_{x*z,y*z}(\phi_{y,z}(b)+\psi_{y,z}(c)).$$
			Putting $b=0_y$ and $c=0_z$, we get \hyperref[M1]{(M1)}. Similarly, taking $a=0_x$ and $c=0_z$ gives \hyperref[M2]{(M2)}, whereas taking $a=0_x$ and $b=0_y$ gives \hyperref[M7]{(M7)}. Conversely, if $(A,\phi_{x,y},\psi_{x,y},\eta_x)$ is an $(X,\rho)$-module and $\sigma$ is a symmetric rack 2-cocycle, then using $(X,\rho)$-module axioms \hyperref[M1]{(M1)}, \hyperref[M2]{(M2)}, \hyperref[M7]{(M7)} and the cocycle condition \hyperref[cocycle condition]{(F1)}, we get
			$$\alpha_{x*y,z} \big(\alpha_{x,y}(a,b),c \big)=	\alpha_{x*z,y*z} \big(\alpha_{x,z}(a,c),\alpha_{y,z}(b,c) \big),$$
which is condition (2) for $(\alpha, \beta)$ to be a dynamical cocycle.
			\item  For all $x,y\in X$ and $a \in A_x,b \in A_y,$ we have 
			$$\alpha_{\rho(x),y} \big(\beta_x(a),b \big)=\phi_{\rho(x),y}(\eta_x(a))+\psi_{\rho(x),y}(b)+\sigma_{\rho(x),y}$$ and 
			$$\beta_{x*y} \big(\alpha_{x,y}(a,b)\big)=\eta_{x*y} \big(\phi_{x,y}(a)+\psi_{x,y}(b)+\sigma_{x,y} \big).$$
Thus, putting appropriate identities of abelian groups as in the argument of $(2)$ above, we see that $\alpha_{\rho(x),y}(\beta_x(a),b)=\beta_{x*y}(\alpha_{x,y}(a,b))$ if and only if $(X,\rho)$-module axioms \hyperref[M4]{(M4)}, \hyperref[M5]{(M5)} and cocycle condition \hyperref[cocycle condition]{(F2)} holds.

\item The condition (4) for $(\alpha, \beta)$ to be a dynamical cocycle is $\beta_{\rho(x)}\beta_x(a)=a$, which holds if and only if $\eta_{\rho(x)}\eta_x(a)=a$ and $(X,\rho)$-module axiom $\hyperref[M3]{(M3)}$ holds for all $x \in X$ and $a \in A_x.$
\item For all $x,y\in X$ and $a \in A_x,b \in A_y,$ we have 
$$\alpha_{x,\rho(y)} \big(\beta_y(b)\big)(a)=\phi_{x,\rho(y)}(a)+\psi_{x,\rho(y)}(\eta_y(b))+\sigma_{x,\rho(y)}$$
and
$$\big(\alpha_{x*^{-1}y,y}(b) \big)^{-1}(a)=\phi^{-1}_{x*^{-1}y,y}(a)-\phi^{-1}_{x*^{-1}y,y}\psi_{x*^{-1}y,y}(b)-\phi^{-1}_{x*^{-1}y,y}\sigma_{x*^{-1}y,y}.$$
The condition (5) for $(\alpha, \beta)$ to be a dynamical cocycle is  $\alpha_{x,\rho(y)}(\beta_y(b))(a)=\alpha_{x*^{-1}y,y}(b))^{-1}(a)$, which holds if and only if \hyperref[M6]{(M7)}, \hyperref[M6]{(M8)} and cocycle condition \hyperref[cocycle condition]{(F3)} holds.
\end{enumerate}
In addition if $(X,\rho)$ is a quandle, then $\alpha_{x,x}(s,s)=s$ if and only if \hyperref[M9]{(M9)} holds and $\sigma_{x,x}=0_x.$ This completes the proof of assertion (A).

\item[]
\item If $\sigma$ and $\sigma'$ are cohomologous, then there exists $\tau:X \rightarrow \cup_{x \in X}A_x$ such that $\tau_x \in A_x$, $\eta_x(\tau(x))=\tau(\rho(x))$ and $\sigma_{x,y}-\sigma'_{x,y}=\phi_{x,y} \tau(x)-\tau(x*y)+\psi_{x,y}(\tau(y))$. For each $x \in X, ~a \in A_x$, we define
$$ \gamma:X \rightarrow \cup_{x \in X}\Sigma_{A_x}~\textrm{by}~			\gamma_x(a)=\tau(x)+a.$$
Clearly, $\gamma_x \in \Sigma_{A_x}$ with its inverse defined as $\gamma_x^{-1}(a)=a-\tau(x)$ for all $a \in A_x.$
	Now, \begin{eqnarray*}
	\alpha'_{x,y}(a,b)&=& \phi_{x,y}(a)+\psi_{x,y}(b)+\sigma'_{x,y}\\
	&=&\gamma_{x*y}\big(\phi_{x,y}(a)+\psi_{x,y}(b)+\sigma'_{x,y}-\tau(x*y)\big), ~\textrm{by definition of}~\gamma\\
	&=&\gamma_{x*y}\big(\phi_{x,y}(a-\tau(x))+\psi_{x,y}(b-\tau(y))+\sigma'_{x,y}+\phi_{x,y} \tau(x)-\tau(x*y)+\psi_{x,y}(\tau(y))\big)\\
	&=&\gamma_{x*y}\big(\phi_{x,y}(\gamma_x^{-1}(a))+\psi_{x,y}(\gamma_y^{-1}(b))+\sigma_{x,y}\big)\\
	&=&\gamma_{x*y}\big( \alpha_{x,y}(\gamma_x^{-1}(a), \gamma_y^{-1}(b))\big)
\end{eqnarray*}
and
$$\beta'_x(a)=\eta_x(a)= \tau(\rho(x))+\eta_x(a)-\eta_x(\tau(x))= \gamma_{\rho(x)} \big(\eta_x(a-\tau(x)) \big)= \gamma_{\rho(x)} \big(\beta_x(\gamma_x^{-1}(a)) \big).
$$
Hence, $(\alpha,\beta)$ is cohomologous to $(\alpha',\beta')$, and the proof is complete.
\end{enumerate}
\end{proof}
\medskip

\section{From group extensions to dynamical extensions}\label{sec group extensions and dynamical extensions}
In this section, we explore connection between group extensions and dynamical extensions. The following result is inspired from \cite[Corollary 2.5]{MR1994219}.

\begin{proposition}\label{surjective homomorphism characterisation}
	Let $(X,\rho)$ and $(\tilde{X},\tilde{\rho})$ be symmetric racks (respectively, quandles). Let $f:(\tilde{X},\tilde{\rho}) \rightarrow (X,\rho)$ be a surjective symmetric rack (respectively, quandle) homomorphism and $T=\{f^{-1}(x) \mid x \in X\}$. Let $S$ be a collection of sets $\{ S_x \mid x \in X \}$ such that there is a bijection $\mu_x:f^{-1}(x)  \rightarrow S_x$ for each $x \in X$. Then $(\tilde{X},\tilde{\rho}) \cong (X \times_{(\alpha,\beta)}S, \rho_{\alpha,\beta})$ for some dynamical cocycle $(\alpha,\beta).$ 

\end{proposition}
\begin{proof} Let $X \times S:=\{(x,s)\mid x \in X~\textrm{and}~s \in S_x\}.$ Clearly, the map $\Phi: X \times S \to \tilde{X}$ given by $\Phi(x, s)=\mu_x^{-1}(s)$, is a bijection.  Define the maps
$$ \alpha:X \times X \rightarrow \cup_{x,y \in X}\Map(S_x \times S_y, S_{x*y})~\textrm{by}~ \alpha_{x,y}(s,t) = \mu_{x*y} \big(\mu_x^{-1}(s) \tilde{*}	\mu_y^{-1}(t) \big)$$
and 
$$ \beta:X \rightarrow \Map\big(S_x,S_{\rho(x)}\big)~\textrm{by}~ \beta_x(s)= \mu_{\rho(x)} \big(\tilde{\rho}(\mu_x^{-1}(s))\big),$$
where $\tilde{*}$ is the rack operation on $\tilde{X}.$ Note that $	\alpha_{x,y}(s,t)\in S_{x*y}$ and $\beta_x(s) \in S_{\rho(x)}.$ Since $(\tilde{X},\tilde{\rho})$ is a symmetric rack, it follows that $(\alpha,\beta)$ is a dynamical cocycle of $(X,\rho)$ over $S$, and the bijection $\Phi$ becomes an isomorphism of symmetric racks. 
\end{proof}

\begin{corollary}\label{conj core extension}
Let $1 \rightarrow A \rightarrow E \rightarrow G \rightarrow 1$ be an extension of groups and $SQ(E)=(\Conj_n(E),\mathrm{inv})$ or $(\Core(E), \id)$. Then $SQ(E) \cong SQ(G) \times_{(\alpha,\beta)} SQ(A)$ for some dynamical cocycle $(\alpha, \beta).$	
\end{corollary}
\begin{proof}
	Let $\pi:E \rightarrow G$ be the quotient homomorphism and $\kappa:G \rightarrow E$ a set-theoretic section such that $\kappa(1)=1$. By functoriality, we have a surjective symmetric quandle homomorphism $SQ(\pi):SQ(E) \rightarrow SQ(G)$ such that 
	$$SQ(\pi)\big(\kappa(x)*\kappa(y)\big)=x*y=SQ(\pi)\big(\kappa(x*y)\big)$$
	for all $x,y \in G$. Thus, there exists a unique element $\theta(x,y) \in A$ such that 
	$$\kappa(x)*\kappa(y)=\kappa(x*y)\theta(x,y).$$
Note that, every element of $E$ can be written uniquely in the form $\kappa(x)s$ for some $x \in G$ and $s \in A$. For each $x \in G,$ there is a bijection $g_x:SQ(\pi)^{-1}(x) \rightarrow A$ given by $g_x(\kappa(x)s)=s.$ In the notation of Proposition \ref{surjective homomorphism characterisation}, we identify each $S_x$ with $A$. Hence, by Proposition \ref{surjective homomorphism characterisation}, we obtain $SQ(E) \cong SQ(G) \times_{(\alpha,\beta)} A$, where $(\alpha,\beta)$ is the dynamical cocycle given by 
$$\alpha_{x,y}(s,t)= g_{x*y} \big(g_x^{-1}(s)\tilde{*}g_y^{-1}(t)\big)~\textrm{and}~ \beta_x(s)=g_{\rho(x)} \big(\tilde{\rho}(g_x^{-1}(s))\big)$$ for $x,y\in G$ and $s,t \in A$.
\end{proof}

\begin{corollary}\label{core extension}
	Let $1 \rightarrow A \rightarrow E \xrightarrow{\pi} G \rightarrow 1$ be an extension of groups, $z \in A$ be a central element of $E$ and $\bar{z}:=\pi(z)$. Then $(\Core(E), z) \cong (\Core(G), \bar{z}) \times_{(\alpha,\beta)} (\Core(A), z)$ for some dynamical cocycle $(\alpha, \beta).$	
\end{corollary}

\begin{proof}
As above, let $\pi:E \rightarrow G$ be the quotient map and $\kappa:G \rightarrow E$ a set-theoretic section such that $\kappa(1)=1$. We have a surjective symmetric quandle homomorphism $\Core(\pi):(\Core(E), z) \rightarrow (\Core(G), \bar{z}))$. Employing the arguments similar to corollary \ref{conj core extension}, we obtain $(\Core(E), z) \cong (\Core(G), \bar{z}) \times_{(\alpha,\beta)} (\Core(A), z)$.
\end{proof}
\medskip

\section{Wells type exact sequence for symmetric racks}\label{sectionwellstype}
Let $(X,\rho_{X})$ be a symmetric rack and $A$ be a left $\mathbb{Z}(X, \rho_{X})$-module. Let $\operatorname{Aut}(X, \rho_{X})$ denote the group of symmetric rack automorphisms of $(X, \rho_{X})$, and $\operatorname{Aut}(A)$ denote the group of $\mathbb{Z}(X, \rho_{X})$-module automorphisms of $A$. For investigations in this section, we will restrict ourselves to the case where $\phi_{x, y}=\phi$, $\psi_{x, y}=\psi$ and $\eta_{x}=\eta$, i.e. when $\phi, \psi, \eta$ do not depend on $x, y \in X$. Note that, even with this restriction, the coefficient system is quite general (see \cite[Example 3.3]{KSS2024}).
\par

For $(\zeta, \theta) \in \operatorname{Aut}(X, \rho_{X}) \times \operatorname{Aut}(A)$ and a 2-cocycle $\alpha \in Z_{S R}^{2}((X, \rho_{X}); A)$, we define
$${}^{(\zeta,\theta)}\alpha_{x,y}\coloneqq\theta \big(\alpha_{\zeta^{-1}(x), \zeta^{-1}(y)} \big).$$

\begin{lemma}
	The map $\sigma: X \times X \longrightarrow A$ given by $\sigma(x, y)=\theta \big(\alpha_{\zeta^{-1}(x), \zeta^{-1}(y)} \big)$ is a symmetric rack 2-cocycle.
\end{lemma}
\begin{proof}
We verify the symmetric rack 2-cocycle conditions as given in Remark \ref{cocycleconditions}. For any $x, y, z \in X$, we have
\begin{small}
\begin{eqnarray*}
			&& -\phi_{x * z, y * z} \sigma(x, z)+\phi_{x * y, z} \sigma(x, y)+\sigma(x * y, z) -\sigma(x * z, y * z)-\psi_{x * z, y * z} \sigma(y, z) \\
			& = &-\phi_{x * z, y * z} \theta(\alpha_{\zeta^{-1}(x), \zeta^{-1}(z)})+\phi_{x * y, z} \theta(\alpha_{\zeta^{-1}(x), \zeta^{-1}(y)}) +\theta(\alpha_{\zeta^{-1}(x * y), \zeta^{-1}(z)}) \\
			&& -\theta(\alpha_{\zeta^{-1}(x * z), \zeta^{-1}(y * z)})  -\psi_{x * z, y * z} \theta(\alpha_{\zeta^{-1}(y), \zeta^{-1}(z)})\\
			& = &\theta \Big(-\phi_{x * z, y * z} \alpha_{\zeta^{-1}(x), \zeta^{-1}(z)} +\phi_{x * y, z} \alpha_{\zeta^{-1}(x), \zeta^{-1}(y)} +\alpha_{\zeta^{-1}(x * y), \zeta^{-1}(z)}\\
			&& -\alpha_{\zeta^{-1}(x * z), \zeta^{-1}(y * z)}-\psi_{x * z, y * z} \alpha_{\zeta^{-1}(y), \zeta^{-1}(z)} \Big),\\
			&& ~\textrm{since $\theta$ is a $\mathbb{Z}(X, \rho_{X})$-module homomorphism}\\
			& = &\theta \Big(-\phi_{\zeta^{-1}(x) * \zeta^{-1}(z), \zeta^{-1}(y) * \zeta^{-1}(z)}  \alpha_{\zeta^{-1}(x) , \zeta^{-1}(z)}  +\phi_{\zeta^{-1}(x) * \zeta^{-1}(y), \zeta^{-1}(z)} \alpha_{\zeta^{-1}(x), \zeta^{-1}(y)}  \\
			&& +\alpha_{\zeta^{-1}(x) * \zeta^{-1}(y), \zeta^{-1}(z)} -\alpha_{\zeta^{-1}(x) * \zeta^{-1}(z), \zeta^{-1}(y) * \zeta^{-1}(z)}  -\psi_{\zeta^{-1}(x) * \zeta^{-1}(z), \zeta^{-1}(y) * \zeta^{-1}(z)} \alpha _{\zeta^{-1}(y), \zeta^{-1}(z)} \Big),\\
& & \text{since $\phi_{x, y}=\phi$ and $\psi_{x, y}=\psi$ for all $x, y \in X$}\\
& = &\theta(0), \quad \text{since $\alpha$ is a symmetric rack 2-cocycle}\\
&=& 0.
\end{eqnarray*}	
\end{small}
By similar arguments as above, we can show that
$$ \eta_{x *y} \sigma(x, y)-\sigma(\rho_{X}(x), y)=0$$
and 
$$\phi_{x, y} \sigma(x * y, \rho_{X}(y))+\sigma(x, y)=0$$
 for any $x, y \in X$. This proves that $\sigma$ is a symmetric rack 2-cocycle.
\end{proof}

\begin{lemma}
	The map
$$
\big(\operatorname{Aut}(X, \rho_{X}) \times \operatorname{Aut}(A) \big) \times Z_{SR}^{2} \big((X, \rho_{X}) ; A \big) \longrightarrow Z_{SR}^{2} \big((X, \rho_{X}) ; A \big)
$$
given by $\big((\zeta, \theta),  \alpha\big) \mapsto ~^{(\zeta,\theta)}\alpha$, where $^{(\zeta,\theta)}\alpha_{x, y}=\theta(\alpha_{\zeta^{-1}(x), \zeta^{-1}(y)})$, defines a left action of the group $\operatorname{Aut}(X, \rho_{X}) \times \operatorname{Aut}(A)$ on the group $Z_{SR}^{2}((X, \rho_{X}) ; A)$ by automorphisms.
\end{lemma}

\begin{proof}
It is easy to see that
\begin{eqnarray*}
		& & \big(\big((\zeta_{1}, \theta_{1}) \cdot(\zeta_{2}, \theta_{2})\big) \cdot \alpha \big)_{x, y} \\
		& =& \big((\zeta_{1} ~ \zeta_{2}, \theta_{1} ~ \theta_{2}) \cdot \alpha \big)_{x, y} \\
		&=& (\theta_{1} ~ \theta_{2})\big(\alpha_{\zeta_{2}^{-1} ~ \zeta_{1}^{-1}(x), \zeta_{2}^{-1} ~ \zeta_{1}^{-1}(y)}\big)\\
		& =&\theta_{1} \big(((\zeta_{2}, \theta_{2}) \cdot \alpha)_{\zeta_{1}^{-1}(x), \zeta_{1}^{-1}(y)} \big) \\
		& =&\theta_{1}\big(\beta_{\zeta_{1}^{-1}(x), \zeta_{1}^{-1}(y)}\big),\quad \text{where we denote $(\zeta_{2}, \theta_{2}) \cdot \alpha$ by $\beta$}\\
		&=& \big((\zeta_{1}, \theta_{1}) \cdot\beta \big)_{x, y}\\
		&=& \big((\zeta_{1}, \theta_{1}) \cdot \big((\zeta_{2}, \theta_{2}) \cdot \alpha \big) \big)_{x, y} 
	\end{eqnarray*}
for all $x, y \in X$. This gives a group homomorphism $\Omega : \operatorname{Aut}(X, \rho_{X}) \times \operatorname{Aut} (A)\longrightarrow \text{Bij}(Z_{SR}^{2}((X, \rho_{X})) ; A)$, where the latter is the group of bijections. It remains to show that the image of $\Omega$ lies in $\text{Aut}(Z_{SR}^{2}((X, \rho_{X})) ; A)$. This can be easily seen as
	\begin{eqnarray*}
		& &\Omega(\zeta,\theta)(\alpha+\beta)_{x,y} \\
		&=&\theta \big((\alpha + \beta )_{\zeta^{-1}(x),\zeta^{-1}(y)} \big)\\
		&=&\theta \big(\alpha_{\zeta^{-1}(x),\zeta^{-1}(y)} + \beta_{\zeta^{-1}(x),\zeta^{-1}(y)} \big)\\
		&=&\theta \big(\alpha_{\zeta^{-1}(x),\zeta^{-1}(y)} \big) + \theta \big(\beta_{\zeta^{-1}(x),\zeta^{-1}(y)} \big) \\
		&=&\Omega(\zeta,\theta)(\alpha)_{x,y}+\Omega(\zeta,\theta)(\beta)_{x,y},
	\end{eqnarray*}
and the proof is complete.
\end{proof}

\begin{proposition}
	The group $\operatorname{Aut}(X, \rho_{X}) \times \operatorname{Aut} (A)$ acts from the left on $H_{SR}^{2}((X, \rho_{X}); A)$ via automorphisms by $\big((\zeta,\theta), [\alpha] \big) \mapsto [{}^{(\zeta,\theta)}\alpha]$ for $(\zeta, \theta) \in \operatorname{Aut}(X, \rho_{X}) \times \operatorname{Aut}(A)$ and $[\alpha] \in H_{S R}^{2}((X,\rho_{X});A)$. 	
\end{proposition}

\begin{proof}
It suffices to prove that ${}^{(\zeta,\theta)}\beta \in B_{SR}^{2}((X, \rho_{X});A)$ for all $(\zeta,\theta) \in \operatorname{Aut}(X, \rho_{X}) \times \operatorname{Aut} (A)$ and $\beta \in B_{SR}^{2}((X, \rho_{X});A).$ Let  $(\zeta,\theta) \in \operatorname{Aut}(X, \rho_{X}) \times \operatorname{Aut} (A)$ and $\beta \in B_{SR}^{2}((X, \rho_{X});A)$. Then, there exists a map $\lambda:X\rightarrow A$ such that $\beta_{x,y}=\lambda(\phi x - x*y + \psi y)$ and $\lambda(\eta z - \rho_X(z))=0 $ for all $x, y, z \in X$. We define $\nu:X\rightarrow A $ by $\nu(x)=\theta(\lambda(\zeta^{-1}(x)))$ and extending it linearly.  Then, we see that
	\begin{eqnarray*}
 {}^{(\zeta,\theta)}\beta_{x, y} &=&\theta \big(\beta _{\zeta^{-1}(x),\zeta^{-1}(y)} \big) \\
		&= &\theta \big(\lambda \big(\phi\zeta^{-1}(x)-\zeta^{-1}(x)*\zeta^{-1}(y)+\psi\zeta^{-1}(y) \big)\big)\\
		&=& \theta\big(\phi\lambda(\zeta^{-1}(x))-\lambda(\zeta^{-1}(x)*\zeta^{-1}(y))+\psi\lambda\zeta^{-1}(y) \big),\\
		& & ~\textrm{since $\lambda$ is a $\mathbb{Z}(X, \rho_{X})$-module homomorphism}\\
		&=&\phi\theta \big(\lambda({\zeta^{-1}(x)})\big) - \theta \big(\lambda({\zeta^{-1}(x*y)})\big) +\psi\theta \big(\lambda({\zeta^{-1}(y)})\big),\\
		& & ~\textrm{since $\theta$ is a $\mathbb{Z}(X, \rho_{X})$-module automorphism of $A$} \\
		&=&\phi\nu(x)-\nu(x*y)+\psi\nu(y)\\ 
		&=&\nu(\phi x-x*y+\psi y)
	\end{eqnarray*}
and
\begin{eqnarray*}
		\nu(\eta z - \rho_X(z)) &=& \eta\nu(z) - \nu(\rho_X(z))\\
		&=&\eta\theta(\lambda(\zeta^{-1}(z))) - \theta(\lambda(\zeta^{-1}(\rho_X(z)))),\\
	& & ~\textrm{since $\lambda$ and $\theta$ are $\mathbb{Z}(X, \rho_{X})$-module homomorphisms}\\
		&=&\theta \big(\lambda \big(\eta\zeta^{-1}(z) - \rho_X(\zeta^{-1}(z)) \big)\big) \\
		&=&\theta(0)\\
		&=& 0.
	\end{eqnarray*}
	Hence, ${}^{(\zeta,\theta)}\beta \in B_{SR}^{2}((X, \rho_{X});A)$, and the proof is complete.
\end{proof}
\par

Let $\alpha \in Z_{S R}^{2}((X, \rho_{X}); A)$ be a fixed symmetric rack 2-cocycle and $\mathcal{F}=(A, \phi, \psi, \eta)$ be an $(X, \rho_{X})$-module. Define
$$
E(\mathcal{F}, \alpha):=\big\{(x, a) \mid x \in X~\textrm{and}~ a \in A_{x} \big\}.
$$
Next, we define a binary operation on $E(\mathcal{F}, \alpha)$ by
$$
(x, a) \, \tilde{*} \,(y, b)=\big( x * y,~ \phi_{x, y}(a)+\psi_{x, y}(b)+\alpha_{x, y} \big)
$$
and a map $\rho_{E(\mathcal{F}, \alpha)}: E(\mathcal{F}, \alpha) \longrightarrow E(\mathcal{F}, \alpha)$ by
$$
\rho_{E(\mathcal{F}, \alpha)}(x, a) =\big(\rho_{X}(x), ~\eta_{x}(a)\big)
$$
for all $x, y \in X$, $a \in A_{x}$ and $b \in A_{y}$. Then, $(E(\mathcal{F}, \alpha), \rho_{E(\mathcal{F}, \alpha)})$ is an extension of $(X, \rho_{X})$ by $\mathcal{F}$ through $\alpha$. We refer to this extension as the \textit{abelian extension} of $(X, \rho_{X})$ by $\mathcal{F}$ through $\alpha$. As stated earlier, we stick to the case where $A_x=A$, $\phi_{x, y}=\phi$, $\psi_{x, y}=\psi$ and $\eta_{x}=\eta$ for all $x, y \in X$.
\par

Let us define
\begin{eqnarray*}
\operatorname{Aut}_{A} \big(E(\mathcal{F}, \alpha),\rho_{E(\mathcal{F}, \alpha)}\big)	&:=& \Big\{\xi \in \operatorname{Aut} \big(E(\mathcal{F}, \alpha),\rho_{E(\mathcal{F}, \alpha)} \big) \mid \xi(x, s)= \big(\zeta(x), \lambda(x) +\theta(s)\big) \\
	&& \text{for some}~ (\zeta, \theta) \in      \operatorname{Aut}(X, \rho_{X}) \times \operatorname{Aut}(A)~\text{and some map}~ \lambda: X \longrightarrow A\Big\}.
\end{eqnarray*}

\begin{lemma}\label{subgrouplemma}	
$\operatorname{Aut}_{A}\big(E(\mathcal{F}, \alpha),\rho_{E(\mathcal{F}, \alpha)}\big)$ is a subgroup of $\operatorname{Aut} \big(E(\mathcal{F}, \alpha),\rho_{E(\mathcal{F}, \alpha)}\big)$.
\end{lemma}

\begin{proof}
	Let $\xi,\xi' \in \operatorname{Aut}_{A}(E(\mathcal{F}, \alpha),\rho_{E(\mathcal{F}, \alpha)})$ be such that $\xi(x, s)=(\zeta(x), \lambda(x) +\theta(s))$ and $\xi'(x, s)=(\zeta'(x), \lambda'(x) +\theta'(s))$ for some $(\zeta, \theta), (\zeta', \theta')  \in      \operatorname{Aut}(X, \rho_{X}) \times \operatorname{Aut}(A)$ and some maps $\lambda, \lambda': X \longrightarrow A$. Then, we have
\begin{equation}\label{product auto formula}
\xi \xi'(x,s) = \xi \big(\zeta'(x), ~\lambda'(x) + \theta'(s)\big) = \big(\zeta\zeta'(x),~ (\lambda\zeta' + \theta\lambda')(x) + \theta\theta'(s)\big),
\end{equation}
which shows that $\xi\xi' \in \operatorname{Aut}_{A}(E(\mathcal{F}, \alpha),\rho_{E(\mathcal{F}, \alpha)})$. Next, we define 
$$\bar{\xi}(x,s)=\big(\zeta^{-1}(x), ~\theta^{-1}(-\lambda(\zeta^{-1}(x))+s) \big).$$  The fact that $\xi$ is a symmetric rack homomorphism gives
	\begin{equation}
		\phi \theta^{-1} \big(-\lambda(\zeta^{-1}(x))\big)+\psi\theta^{-1}\big(-\lambda(\zeta^{-1}(y))\big)+\alpha_{\zeta^{-1}(x),\zeta^{-1}(y)}=\theta^{-1} \big(-\lambda(\zeta^{-1}(x*y))\big)+\theta^{-1}(\alpha_{x,y})\label{eq1} 
	\end{equation} 
and
	\begin{equation}
		\theta^{-1} \big(\lambda(\zeta^{-1}(\rho_X(x)))\big) = \eta\big(\theta^{-1}(\lambda(\zeta^{-1}(x)))\big).
		\label{eq2}	
	\end{equation}
We now consider
	\begin{eqnarray*}
&& \bar{\xi} \big((x,a)\, \tilde{*} \,(y,b) \big)\\
		&=& \bar{\xi} \big(x*y, ~\phi a + \psi b + \alpha_{x,y} \big) \\
		&=& \big(\zeta^{-1}(x*y), ~\theta^{-1}(-\lambda(\zeta^{-1}(x*y)) + \phi a + \psi b + \alpha_{x,y}) \big) \\
		&=& \big(\zeta^{-1}(x*y), ~\theta^{-1}(-\lambda \zeta^{-1}(x*y)) + \theta^{-1}(\alpha_{x,y}) + \phi \theta^{-1}(a) + \psi \theta^{-1}(b) \big) \\
		&=& \big(\zeta^{-1}(x*y), ~\phi \theta^{-1}(-\lambda(\zeta^{-1}(x))) + \psi \theta^{-1}(-\lambda(\zeta^{-1}(y))) + \alpha_{\zeta^{-1}(x), \zeta^{-1}(y)}+ \phi \theta^{-1}(a) + \psi \theta^{-1}(b) \big),\\
		&& ~\text{using \eqref{eq1}} \\
		&=& \big(\zeta^{-1}(x)*\zeta^{-1}(y), ~\phi \theta^{-1}(-\lambda({\zeta^{-1}(x)}) + a) + \psi \theta^{-1}(-\lambda({\zeta^{-1}(y)}) + b) + \alpha_{\zeta^{-1}(x), \zeta^{-1}(y)}\big) \\
		&=& \big(\zeta^{-1}(x), \theta^{-1}(-\lambda({\zeta^{-1}(x)}) + a)\big) \, \tilde{*} \, \big(\zeta^{-1}(y), \theta^{-1}(-\lambda({\zeta^{-1}(y)}) + b) \big) \\
		&=& \bar{\xi}(x,a) \, \tilde{*} \, \bar{\xi}(y,b)
	\end{eqnarray*}
and
\begin{eqnarray*}
&& \bar{\xi}\big(\rho_{E(\mathcal{F}, ~\alpha)}(x,a) \big) \\
		&=& \bar{\xi} \big(\rho_X(x), \eta a \big) \\
		&=& \big(\zeta^{-1}(\rho_X(x)), ~\theta^{-1}(-\lambda(\zeta^{-1}(\rho_X(x)))+ \eta a)\big)\\
		&=&\big( \rho_X(\zeta^{-1}(x)), ~\eta(\theta^{-1}(-\lambda(\zeta^{-1}(x))))+\eta \theta^{-1}(a) \big),\\
		 && \text{using \eqref{eq2} and the fact that $\theta^{-1}$ is a $\mathbb{Z}(X, \rho_{X})$-module homomorphism}\\
		&=&\big( \rho_X(\zeta^{-1}(x)), ~\eta(\theta^{-1}(-\lambda(\zeta^{-1}(x))+a))\big)\\
		&=&\rho_{E(\mathcal{F}, ~\alpha)} \big(\zeta^{-1}(x), ~\theta^{-1}(-\lambda(\zeta^{-1}(x))+a)\big)\\
		&=&\rho_{E(\mathcal{F}, ~\alpha)} \big(\bar{\xi}(x,a) \big).
	\end{eqnarray*}
It is easy to see that $\bar{\xi}$ is bijective, and hence $\bar{\xi} \in \operatorname{Aut}_{A}(E(\mathcal{F}, \alpha),\rho_{E(\mathcal{F}, \alpha)}).$ Further, a direct calculation shows that $\bar{\xi}$ and $\xi$ are indeed inverses of each other, which completes the proof of the lemma.
\end{proof}

The discussion and observations that follow, serve to prepare the ground for the upcoming proposition. Note that, there is a natural restriction map

$$
\Gamma: \operatorname{Aut}_{A}(E(\mathcal{F}, \alpha),\rho_{E(\mathcal{F}, \alpha)}) \longrightarrow \operatorname{Aut}(X, \rho_{X}) \times \operatorname{Aut}(A)
$$
given by $\Gamma(\xi)=(\zeta, \theta)$, where $\xi(x, s)=(\zeta(x), \lambda(x) +\theta(s))$. It follows from \eqref{product auto formula} that $\Gamma$ is a group homomorphism.
\par

We now define another map
$$
\Lambda_{[\alpha]}: \operatorname{Aut}(X, \rho_{X}) \times \operatorname{Aut}(A) \longrightarrow H_{SR}^{2}((X, \rho_{X});A). 
$$
Given a pair $(\zeta, \theta) \in \operatorname{Aut}(X,\rho_X) \times \operatorname{Aut}(A)$, consider the elements  ${}^{(\zeta,\theta)}[\alpha]$ and $[\alpha]$ of $H_{SR}^{2}((X, \rho_{X});A)$. Since the group $H_{S R}^{2}((X, \rho_{X}) ; A)$ has a free and transitive left action on itself by multiplication, there is a unique element $\Lambda_{[\alpha]}(\zeta, \theta) \in H_{S R}^{2}((X, \rho_{X}) ; A)$ such that
$$
{}^{\Lambda_{[\alpha]}(\zeta, \theta)}\big({}^{(\zeta, \theta)}[\alpha] \big)=[\alpha]. 
$$
Let us set
$$
\big(\operatorname{Aut}(X, \rho_{X}) \times \operatorname{Aut}(A)\big)_{[\alpha]} := \big\{(\zeta, \theta) \in \operatorname{Aut}(X, \rho_{X}) \times \operatorname{Aut}(A) \mid {}^{(\zeta,\theta)}[\alpha]=[\alpha] \big\}.		
$$

\begin{proposition}
The maps $\Gamma$  and $\Lambda_{[\alpha]}$ satisfy 
$$\operatorname{Im}(\Gamma)=\big(\operatorname{Aut}(X, \rho_{X}) \times\operatorname{Aut}(A)\big)_{[\alpha]}=\Lambda_{[\alpha]}^{-1}(0),$$ where $0$ is the identity element of $H_{SR}^{2}((X,\rho_{X});A)$.
\end{proposition}

\begin{proof}
It follows from the definition of $\Lambda$ that $(\operatorname{Aut}(X, \rho_{X}) \times\operatorname{Aut}(A))_{[\alpha]}=\Lambda_{[\alpha]}^{-1}(0)$. Suppose that $(\zeta,\theta)\in(\operatorname{Aut}(X, \rho_{X}) \times\operatorname{Aut}(A))_{[\alpha]}$. Then, there exists a map $\lambda: X \longrightarrow A$ such that
	\begin{equation}
		\big({}^{(\zeta,\theta)}\alpha-\alpha \big)(x, y)=\lambda(\phi x-x *y+\psi y)
		\label{eq3}
	\end{equation}	
and
	\begin{equation}
		\lambda(\eta z-\rho_{X}(z))=0
		\label{eq4}
	\end{equation}
for all $x, y, z \in X$. Define $\xi: (E(\mathcal{F}, \alpha),\rho_{E(\mathcal{F}, \alpha)}) \longrightarrow(E(\mathcal{F}, \alpha),\rho_{E(\mathcal{F}, \alpha)})$ by
$$
\xi(x,s)=\big(\zeta(x),~\lambda(\zeta(x))+\theta(s)\big) =\big(\zeta(x),~\lambda'(x)+\theta(s)\big), 
$$
where $\lambda' :=\lambda \zeta: X \longrightarrow A.$	It is easy to see that $\xi$ is a bijection. Also, it follows from \eqref{eq3} and \eqref{eq4} that $\xi$ is indeed a symmetric rack homomorphism. Hence,  $\xi\in \operatorname{Aut}_{A}(E(\mathcal{F}, \alpha),\rho_{E(\mathcal{F}, \alpha)})$ and $\Gamma(\xi)=(\zeta,\theta)$, which helps us conclude that $(\operatorname{Aut}(X, \rho_{X}) \times\operatorname{Aut}(A))_{[\alpha]}\subseteq \operatorname{Im}(\Gamma)$.
\par

For the converse, we take an element $(\zeta,\theta)\in \operatorname{Im}(\Gamma).$ Then there exists $\xi \in \operatorname{Aut}_{A}(E(\mathcal{F}, \alpha),\rho_{E(\mathcal{F}, \alpha)})$ such that $\Gamma(\xi)=(\zeta,\theta)$. By definition of $\Gamma$,  there exists a map  $\lambda: X \longrightarrow A$ such that 
$$\xi(x,s)=\big(\zeta(x),\lambda(\zeta(x))+\theta(s)\big).$$ Now, the fact that $\xi$ is a symmetric rack homomorphism, gives us \eqref{eq3} and \eqref{eq4}. This shows that ${}^{(\zeta,\theta)}[\alpha]=[\alpha]$, and hence $ \operatorname{Im}(\Gamma)\subseteq (\operatorname{Aut}(X, \rho_{X}) \times\operatorname{Aut}(A))_{[\alpha]}$.
\end{proof}

\begin{proposition}
The groups $\operatorname{Ker}(\Gamma)$ and $ Z_{SR}^{1}((X, \rho_{X}); A)$ are isomorphic, where 
	$$Z_{SR}^{1} \big((X, \rho_{X}); A\big)= \big\{\lambda: X \longrightarrow A \mid \lambda(\phi x-x * y+\psi y)=0~ \text{and}~\lambda(\eta z-\rho_{X}(z))=0~\textrm{for all}~ z,  x, y \in X \big\}.$$
\end{proposition} 

\begin{proof}
It is easy to see that
	\begin{eqnarray*}
		\operatorname{Ker}(\Gamma) &=& \left\{\xi ~\Big \vert ~
		\begin{array}{l}
			\xi: (E(\mathcal{F}, \alpha), \rho_{E(\mathcal{F}, \alpha)}) \longrightarrow (E(\mathcal{F}, \alpha), \rho_{E(\mathcal{F}, \alpha)})~ \text{is a symmetric rack homomorphism} \\
			\text{of the form } \xi(x, s) = (x,\lambda(x) + s) \text{ for some } \lambda: X \longrightarrow A
		\end{array} 
		\right\} \\[1em]
		&=& \left\{\xi ~\Big \vert ~ 
		\begin{array}{l}
			\xi: (E(\mathcal{F}, \alpha), \rho_{E(\mathcal{F}, \alpha)}) \longrightarrow (E(\mathcal{F}, \alpha), \rho_{E(\mathcal{F}, \alpha)}) ~\text{is a map of the form} \\
			\xi(x, s) = (x,\lambda(x) + s)~ \text{ for some } ~\lambda: X \longrightarrow A~ \text{which satisfies}~\\
 \lambda(\phi x - x*y + \psi y)=0~\text{ and}~ \lambda(\eta z - \rho_{X}(z))=0~  \text{ for all}~ x, y, z \in X\\
		\end{array} 
		\right\} \\[1em]
		&=& \left\{\xi ~\Big \vert ~ 
		\begin{array}{l}
			\xi: (E(\mathcal{F}, \alpha), \rho_{E(\mathcal{F}, \alpha)}) \longrightarrow (E(\mathcal{F}, \alpha), \rho_{E(\mathcal{F}, \alpha)})~ \text{is a map of the form} \\
			\xi(x, s) = (x,\lambda(x) + s) ~\text{ for some 1-cocycle } ~\lambda:X \to A
		\end{array} 
		\right\}.
	\end{eqnarray*}
We define a map $\Theta:\operatorname{Ker}(\Gamma) \rightarrow Z_{SR}^{1}((X, \rho_{X}); A)$ by $\Theta(\xi)=\lambda$. It is straightforward to check that $\Theta$ is a bijection. Further, $\Theta$ is a homomorphism of groups follows from \eqref{product auto formula}, which completes the proof.
\end{proof}

With the necessary components developed, we can now assemble them to obtain our main theorem. The result yields a four-term exact sequence, which we will refer to as a Wells-type exact sequence, named after Wells \cite{CharlesWells71}, who established a similar exact sequence for group extensions.

\begin{theorem} \label{MainTheoremWells}
Let $(X, \rho_{X})$ be a symmetric rack, $A$ be a $\mathbb{Z}(X, \rho_{X})$-module, $\alpha \in Z_{S R}^{2}((X, \rho_{X}) ; A)$ and $E(\mathcal{F}, \alpha)$ the abelian extension of $(X, \rho_{X})$ by $\mathcal{F}$ through $\alpha$. Then, there is an exact sequence of groups
$$ 0\to Z_{SR}^{1} \big((X, \rho_{X});A \big)\to \operatorname{Aut}_{A} \big(E(\mathcal{F}, \alpha), \rho_{E(\mathcal{F}, \alpha)} \big) \stackrel{\Gamma}{\to}  \operatorname{Aut} (X, \rho_{X})\times\operatorname{Aut}(A) \stackrel{\Lambda_{[\alpha]}}{\to}  H_{SR}^{2} \big((X,\rho_{X});A \big). 
	$$
Here, $\Lambda_{[\alpha]}$ is a set-theoretic map, and the exactness at $\operatorname{Aut}((X, \rho_{X}))\times\operatorname{Aut}(A)$ means that $\operatorname{Im}(\Gamma)=\Lambda_{[\alpha]}^{-1}(0)$.
\end{theorem}

\begin{corollary}
Let $(X, \rho_{X})$ be a symmetric rack, $A$ be a $\mathbb{Z}(X, \rho_{X})$-module, $\alpha \in Z_{S R}^{2}((X, \rho_{X}) ; A)$ and $E(\mathcal{F}, \alpha)$ the abelian extension of $(X, \rho_{X})$ by $\mathcal{F}$ through $\alpha$. Then, there is a short exact sequence of groups
	$$
	0\to Z_{SR}^{1}\big((X, \rho_{X});A \big)\to \operatorname{Aut}_{A} \big(E(\mathcal{F}, \alpha), \rho_{E(\mathcal{F}, \alpha)}\big) \stackrel{\Gamma}{\to}  \big(\operatorname{Aut}(X, \rho_{X}) \times\operatorname{Aut}(A) \big)_{[\alpha]} \to 0.
	$$
\end{corollary}	
	
\begin{remark}\label{obstruction remark}
Suppose that $(X, \rho_{X})$ is a symmetric rack and $A$ is a $\mathbb{Z}(X, \rho_{X})$-module such that $H_{SR}^{2}((X,\rho_{X});A)=0$. Then, we conclude from Theorem \ref{MainTheoremWells} that any pair of automorphisms $(\zeta,\theta)  \in \operatorname{Aut}(X, \rho_{X}) \times\operatorname{Aut}(A)$ can be extended to an automorphism $\xi  \in \operatorname{Aut}_{A}(E(\mathcal{F}, \alpha), \rho_{E(\mathcal{F}, \alpha)})$. Thus, the obstruction to this kind of extension of a pair of automorphisms lies in the second symmetric rack cohomology.
\end{remark}

Let us illustrate the preceding theorem through two examples. In each of them, $A$ denotes a $\mathbb{Z}(X, \rho_{X})$-module with $A_x=\mathbb{Z}$, $\phi_{x, y}=\id$, $\psi_{x, y}=0$ and $\eta_{x}=-\id$ for all $x, y \in X$.

\begin{example} 
Consider the symmetric rack $(X, \rho_{X})$, where $X$ is any trivial rack with $|X|\geq 2$, that is, $x*y=x$ for all $x,y \in X$. Let $\rho_X$ be any non-trivial involution on $X$. Choose $(\zeta, \theta) \in \operatorname{Aut}(X, \rho_{X}) \times \operatorname{Aut}(A)$ such that $\zeta$ is any involution (for instance, we can take $\zeta = \rho_X$) and $\theta = -\id$.  In this scenario, two symmetric rack 2-cocycles $\beta$ and $\gamma$ are cohomologous if and only if  $\beta=\gamma$. Suppose that $\alpha: X \times X \rightarrow A$ is a map such that $\alpha_{x_0,y_0} \neq 0$ for some $x_0,y_0 \in X$ and $$\alpha_{\zeta(x),\zeta(y)}=\alpha_{x,y}=-\alpha_{\rho_X(x),y}=-\alpha_{x,\rho_X(y)}$$ for all $x,y\in X$. A straightforward check using Remark \ref{cocycleconditions} tells us that $\alpha \in Z_{SR}^{2}((X, \rho_{X}); A)$. But, we have
	$$
		{}^{(\zeta, \theta)}\alpha_{x, y} = \theta \big(\alpha_{\zeta^{-1}(x), \zeta^{-1}(y)} \big) = -\alpha_{\zeta(x), \zeta(y)}  = -\alpha_{x, y}
	$$
for all $x, y \in X$.	 Since $\alpha_{x_0,y_0} \neq 0$, we have $-\alpha \neq \alpha$, and hence ${}^{(\zeta, \theta)}\alpha$ and $\alpha$ are not cohomologous. This implies that $\Lambda_{[\alpha]}(\zeta,\theta) \neq 0$, and hence $(\zeta,\theta)$ cannot be extended to an automorphism lying in $\operatorname{Aut}_{A} \big(E(\mathcal{F}, \alpha), \rho_{E(\mathcal{F}, \alpha)} \big)$.
\par
For instance, we can take the trivial rack $X= \{ 0, 1 \}$ with $\rho_X(x)=x+1 \mod 2$ for all $x \in X$. Further, take $(\zeta, \theta) \in \operatorname{Aut}(X, \rho_{X}) \times \operatorname{Aut}(A)$ such that $\zeta = \id$ and $\theta = -\id$, and take $\alpha: X \times X \rightarrow A$ given by
	 $$	 \alpha_{0,0} = 1, \quad \alpha_{1,0} = -1, \quad \alpha_{0,1} = -1 \quad \textrm{and} \quad \alpha_{1, 1} = 1. $$
\end{example}

\begin{example}
Let $(X,\rho_X)$ be a symmetric rack such that $\rho_X$ is the identity map on $X$ (for example,  we can take $X=\mathbb{Z}_n$ with $x*y=2y-x \mod n$ for any $n \geq 2$. Then using condition \eqref{cocyclecondition2.1} of Remark \ref{cocycleconditions}, it is easy to see that any $\alpha \in Z_{SR}^{2}((X, \rho_{X}); A)$ satisfies $\alpha=0$, which further implies that $H_{SR}^{2} \big((X,\rho_{X});A \big)=0$. Thus, every pair of automorphisms $(\zeta,\theta) \in \operatorname{Aut}(X, \rho_{X}) \times\operatorname{Aut}(A)$ extends to an automorphism $\xi \in \operatorname{Aut}_{A}(E(\mathcal{F}, \alpha), \rho_{E(\mathcal{F}, \alpha)})$.

\end{example}

\medskip	

An analogue of Theorem \ref{MainTheoremWells} follows for symmetric quandles. 

\begin{theorem} \label{MainTheoremWellsQuandles}
Let $(X, \rho_{X})$ be a symmetric quandle, $A$ be a $\mathbb{Z}(X, \rho_{X})$-module, $\alpha \in Z_{S Q}^{2}((X, \rho_{X}) ; A)$ and $E(\mathcal{F}, \alpha)$ the abelian extension of $(X, \rho_{X})$ by $\mathcal{F}$ through $\alpha$. Then, there is an exact sequence of groups
$$ 0\to Z_{SQ}^{1} \big((X, \rho_{X});A \big)\to \operatorname{Aut}_{A} \big(E(\mathcal{F}, \alpha), \rho_{E(\mathcal{F}, \alpha)} \big) \stackrel{\Gamma}{\to}  \operatorname{Aut} (X, \rho_{X})\times\operatorname{Aut}(A) \stackrel{\Lambda_{[\alpha]}}{\to}  H_{SQ}^{2} \big((X,\rho_{X});A \big).$$
\end{theorem}

As in Remark \ref{obstruction remark}, the obstruction to extension of automorphisms in abelian extensions of symmetric quandles lies in the second symmetric quandle cohomology.
\medskip

\begin{ack}
BK thanks CSIR for the PhD research fellowship. DS thanks IISER Mohali for the PhD research fellowship. MS is supported by the SwarnaJayanti Fellowship
grants DST/SJF/MSA-02/2018-19 and SB/SJF/2019-20/04.
\end{ack}

\section{Declaration}
The authors declare that there is no data associated to this paper and that there are no conflicts of interests.

\end{document}